\documentclass[a4paper,11pt]{article}

\usepackage[top=3.0cm, bottom=3.0cm, inner=3.0cm, outer=3.0cm, includefoot]{geometry}

\usepackage{verbatim}

\usepackage{geometry}
\usepackage{amssymb}
\usepackage{amsmath}
\usepackage{graphicx}
\usepackage{amsthm}
\usepackage{bbm}
\usepackage{mathrsfs}

\usepackage{color,soul}

\usepackage[T1]{fontenc}
\usepackage[utf8]{inputenc}
\usepackage{authblk}

\usepackage{enumerate}
\usepackage{tikz}
\usetikzlibrary{arrows}
\usetikzlibrary{decorations.markings}

\newcommand{\Ric}{\operatorname{Ric}}

\newcommand{\IR}{{\mathbb{R}}}

\newcommand{\DD}{{\mathscr{D}}}

\newcommand{\Hcal}{{\mathcal{H}}}
\newcommand{\Wcal}{{\mathcal{W}}}
\newcommand{\Bcal}{{\mathcal{B}}}
\newcommand{\Acal}{{\mathcal{A}}}

\newcommand{\rhohat}{\hat{\rho}}

\newcommand{\dd}{\operatorname{d}}

\newcommand{\eChar}{\begin{enumerate}[(i)]}
\newcommand{\eCharR}{\begin{enumerate}[(a)]}
\newcommand{\eBr}{\begin{enumerate}[(1)]}

\newcommand{\diam}{\operatorname{diam}}

\title
{A note on a Bonnet-Myers type diameter bound for graphs with positive entropic Ricci curvature}

\author[1]{S. Kamtue}
\affil[1]{Department of Mathematical Sciences, Durham University}

\date{\today}

\theoremstyle{plain}
\newtheorem{lemma}{Lemma}[section]
\newtheorem{theorem}[lemma]{Theorem}

\newtheorem{corollary}[lemma]{Corollary}

\theoremstyle{definition}

\newtheorem{definition}{Definition}[section]

\newtheorem{rem}[lemma]{Remark}

\renewenvironment{proof}[1][\proofname]{{\bfseries #1.}}{\qed}

\numberwithin{equation}{section}

\begin{document}

\maketitle

\begin{abstract}
An equivalent definition of entropic Ricci curvature on discrete spaces was given in terms of the global gradient estimate in \cite[Theorem 3.1]{EF18}. With a particular choice of the density function $\rho$, we obtain a localized gradient estimate, which in turns allow us to apply the same technique as in \cite{LMP18} to derive a Bonnet-Myers type diameter bound for graphs with positive entropic Ricci curvature. However, the case of the hypercubes indicates that the bound may be not optimal (where $\theta$ is chosen to be logarithmic mean by default). If $\theta$ is arithmetic mean, the Bakry-\'Emery criterion can be recovered and the diameter bound is optimal as it can be attained by the hypercubes. 
\end{abstract}

\section{Introduction} \label{sec:intro}
The notion of entropic Ricci curvature on discrete spaces (i.e., Markov chains, or graphs) was introduced by Erbar and Maas \cite{EM12} inspired by work of Sturm \cite{Sturm} and Lott and Villani \cite{LV} which describes the lower bound of Ricci curvature via the displacement convexity of the entropy functional. Equivalent definitions of Entropic curvature are also given in terms of Bochner's inequality \cite{EM12} and in terms of gradient estimates \cite{EF18}. Detailed notations and definitions are given in Section \ref{sec:setup} and proofs are presented in Section \ref{sec:proof}.

\begin{definition}[Entropic curvature] \label{defn:entropic}
An irreducible and reversible Markov kernel $Q$ on a finite discrete space $X$ (with a steady state $\pi$) \textbf{has entropic curvature at least $\kappa \in \IR$}, written as $\Ric(Q)\ge \kappa$, if and only if the entropy functional $\Hcal$ satisfies the following $\kappa$-convexity inequality:
\begin{equation}
\Hcal(\rho_t) \le (1-t)\Hcal(\rho_0)+t\Hcal(\rho_1)-\frac{\kappa}{2}t(1-t)\Wcal(\rho_0,\rho_1)^2
\end{equation}
for every constant speed geodesic $(\rho_t)_{t\in[0,1]}$ in the (modified) $L^2$-Wasserstein space of probabaility densities $(\DD(X),\Wcal)$.
\end{definition}

\begin{theorem}[Equivalent definitions] \label{thm:equivalent}
For a finite set $X$ equipped with an irreducible and reversible Markov kernel $Q$ (and the steady state $\pi$), the following properties are equivalent:
\begin{itemize}
	\item[$(1)$] $\Ric(Q) \ge \kappa$;
	\item[$(2)$] Bochner's inequality holds for all $\rho \in \DD(X)$ and $f\in C(X)$: 
	\begin{equation} \label{eq:bochner}
	\Bcal(\rho,f) \ge \kappa \Acal(\rho,f),
	\end{equation}
	where $$\Acal(\rho,f):= \langle \hat{\rho} \cdot \nabla f, \nabla f \rangle_{\pi},$$
	and $$\Bcal(\rho,f):= \frac{1}{2}\langle \hat{\Delta}\rho\cdot \nabla f, \nabla f\rangle_{\pi} - \langle \hat{\rho} \cdot \nabla f, \nabla \Delta f \rangle_{\pi},$$ where 
	\begin{align*}
	\hat{\rho}(x,y) &:= \theta(\rho(x), \rho(y)),\\
	\hat{\Delta}\rho &:= \partial_1 \theta (\rho(x), \rho(y))\Delta \rho(x) + \partial_2 \theta (\rho(x), \rho(y))\Delta \rho(y),
	\end{align*}
	and $\theta$ is a logarithmic mean: $\theta(a,b)=\frac{a-b}{\log a - \log b}$;
	\item[$(3)$] Global gradient estimate holds for all $\rho \in \DD(X)$ and $f\in C(X)$ and $t\ge 0$,
	\begin{equation*} 
	|\nabla P_tf|^2_\rho \le e^{-2\kappa t} |\nabla f|^2_{P_t\rho},
	\end{equation*}
	or more explicitly
	\begin{equation} \label{eq:global_grad_est}
	\frac{1}{2} \sum_{u,v} (P_tf(v)-P_tf(u))^2 \hat{\rho}(u,v) Q(u,v)\pi(u) \le
	e^{-2\kappa t} \frac{1}{2} \sum_{u,v} (f(v)-f(u))^2 \widehat{P_t\rho}(u,v) Q(u,v)\pi(u)
	\end{equation}
	where $P_t=e^{t\Delta}$ denotes the heat semigroup.
\end{itemize}
\end{theorem}

The contribution of this paper is to follow ideas from \cite{EF18} to derive a similar local gradient estimate, and to apply the same technique as in \cite{LMP18} to prove a Bonnet-Myers type diameter bound on the underlying graph $X$ (when $Q$ is a simple random walk) in the case of a positive lower bound of entropic curvature.

\begin{theorem}[Local gradient estimate] \label{thm:local_grad_est}
	If $\Ric(Q) \ge \kappa$, then for all $x\in X$ and $f\in C(X)$ and all $\varepsilon\ge 0$ we have
	\begin{equation} \label{eq:local_grad_est}
	\Gamma(P_tf)(x)\pi(x)  \le 
	\dfrac{e^{-2\kappa t}}{2\theta(1,\varepsilon)}[P_t\Gamma(f)(x)\pi(x)+\varepsilon \sum_{y\in N(x)}P_t\Gamma(f)(y)\pi(y)].
	\end{equation}
	where $\Gamma(f)(x):= \sum\limits_{y} (f(y)-f(x))^2Q(x,y)$. 
\end{theorem}

\begin{theorem}[Diameter bound] \label{thm:BonnetMyers}
Let $Q$ represent a simple random walk on $X$ with strictly positive entropic Ricci curvature $\Ric(Q)\ge \kappa>0$. Then the diameter is bounded from above by $\diam(X) \le \frac{2}{\kappa}\sqrt{\frac{D\log D}{D-1}}$ where $D$ is the maximal (vertex) degree.
\end{theorem}

\begin{rem}
It is known that the entropic curvature of (a simple random walk on) the discrete hypercube $\mathcal{Q}^n$ is $\frac{2}{n}$ (see \cite[Example 5.7]{EM12}). Therefore, in view of the hypercubes, the bound in Theorem \ref{thm:BonnetMyers} is not optimal: $$n=\diam(\mathcal{Q}^n)\le n\sqrt{\frac{n}{n-1} \log n}.$$
However, if we replace the logarithmic mean by the arithmetic mean for $\theta$ in Bochner's fomula \eqref{eq:bochner} and inequality \eqref{eq:global_grad_est}, the local gradient estimate in Theorem \ref{thm:local_grad_est} would imply (by taking $\varepsilon=0$) the Bakry-\'Emery curvature criterion $CD(\kappa, \infty)$: $\Gamma(P_t f) (x) \le  e^{-2\kappa t} P_t \Gamma(f)(x)$ in the sense of \cite[Corollary 3.3]{LL15}. This implication has already been mentioned in the survey \cite{Msurvey}. Consequently, we obtain a diameter bound: $\diam(X) \le \frac{2}{\kappa}$, which is sharp and the equality is attained if and only if $X$ is a hypercube $\mathcal{Q}^n$ (for details see \cite{LMP18} and \cite{rigidityBE}).
\end{rem}

\section{Setup and notations}
\label{sec:setup}

\subsection{Notions associated to a discrete Markov chain}

We start with a Markov chain $(X,Q)$, where $X$ is a finite set and $Q:X\times X \rightarrow \IR^+ \cup \{0\}$ is a Markov kernel $Q$, i.e. $\sum_{y\in X} Q(x,y)=1$ for all $x\in X$. Furthermore, we assume that $Q$ is irreducible and reversible, which implies that there exists a unique stationary probability measure $\pi$ on $X$ satisfying $\sum_{x\in X} \pi(x)=1$ and the detail balanced equations: $$Q(x,y)\pi(x)=Q(y,x)\pi(y) \qquad \forall x,y\in X.$$

The set of probability densities (with respect to $\pi$) is defined as
$$\mathscr{D}(X):=\left\{\rho: X \rightarrow \IR^{+}\cup \{0\} \bigg| \sum_{x\in X}\pi(x)\rho(x)=1 \right\}.$$
The entropy functional, defined on $\mathscr{D}(X)$, is given by 
$$\Hcal(\rho):=\sum_{x\in X} \rho(x)\log\rho(x)\pi(x).$$

The \textit{discrete gradient} $\nabla: C(X) \rightarrow C(X\times X)$, \textit{discrete divergence} $\nabla\cdot: C(X\times X) \rightarrow C(X)$, and \textit{laplacian} $\Delta := \nabla\cdot\nabla$ are defined as follows:
\begin{definition}
	For all $f,g\in C(X)$ and $U,V\in C(X\times X)$,
	\begin{align*}
	\nabla f(x,y) &:= f(y)-f(x), \\
	\nabla\cdot V(x) &:= \frac{1}{2} \sum_{y\in X} (V(x,y)-V(y,x))Q(x,y),\\
	\Delta f(x) &:= \sum_{y\in X} (f(y)-f(x))Q(x,y).
	\end{align*}
	The inner products are defined as
	\begin{align*}
	\langle f,g \rangle_{\pi} &:= \sum_{x\in X} f(x)g(x)\pi(x), \\
	\langle U,V \rangle_{\pi} &:= \frac{1}{2}\sum_{x,y\in X} U(x,y)V(x,y)Q(x,y)\pi(x),
	\end{align*}
	and for all $\rho\in \DD(X)$,
	\begin{align*}
	\langle U,V \rangle_{\rho} := \langle \rhohat\cdot U,V\rangle_{\pi} = \frac{1}{2}\sum_{x,y\in X} U(x,y)V(x,y)\rhohat(x,y)Q(x,y)\pi(x),
	\end{align*}
	where $\hat{\rho}(x,y) := \theta(\rho(x), \rho(y))$ and $\theta$ is a \textit{suitable mean} satisfying Assumption 2.1 in \cite{EM12}. By default, $\theta$ is chosen to be the logarithmic mean: $\theta(a,b)=\frac{a-b}{\log a - \log b}$.
	
	Furthermore, we already introduced
	\begin{align*}
	\Acal(\rho,f) &:= |\nabla f|^2_\rho= \langle \hat{\rho} \cdot \nabla f, \nabla f \rangle_{\pi},\\
	\Bcal(\rho,f) &:= \frac{1}{2}\langle \hat{\Delta}\rho\cdot \nabla f, \nabla f\rangle_{\pi} - \langle \hat{\rho} \cdot \nabla f, \nabla \Delta f \rangle_{\pi},
	\end{align*} where $\hat{\Delta}\rho := \partial_1 \theta (\rho(x), \rho(y))\Delta \rho(x) + \partial_2 \theta (\rho(x), \rho(y))\Delta \rho(y).$

\end{definition}

\begin{definition}[Discrete transport metric]
	For $\bar{\rho}_0,\bar{\rho}_1\in \DD(X)$,
	\begin{equation*}
	\Wcal(\bar{\rho}_0,\bar{\rho}_1):=\inf \left\{ \int_0^1 \Acal(\rho_t,f_t) \dd t \ \bigg| \ 
	(\rho_t,f_t) \in CE(\bar{\rho}_0,\bar{\rho}_1) \right\}^{\frac{1}{2}},
	\end{equation*}
	where the infimum is taken over the set $CE(\rho_0,\rho_1)$ which consists of all sufficiently regular curves $(\rho_t)_{t\in[0,1]}$ on $\DD(X)$ and $(f_t)_{t\in[0,1]}$ on $C(X)$ which satisfy the \emph{continuity equation} $\partial_t \rho_t+ \nabla\cdot(\rhohat_t\nabla f_t)=0$ and $\rho_0=\bar{\rho}_0, \rho_1=\bar{\rho}_1$. It was shown that $\Wcal$ is a metric on $\DD(X)$. We refer to \cite{EM12, MaasGF} for further details. Note that this notion of the $\Wcal$-metric is relevant in Definition \ref{defn:entropic} of entropic curvature.

\end{definition}

\subsection{Graph theoretical notions}
The kernel $Q$ induces a graph structure on $X$ by assigning an edge $x\sim y$ if and only if $Q(x,y)>0$. Note that the graph is connected and undirected, due to irreducibility and reversibility of $Q$, respectively. For a vertex $x\in X$, denote $N(x):=\{y\in X| \ x\sim y \}$ the set of neighbors of $x$, and $d_x:=|N(x)|$ the degree of $x$, and $D:=\max_{x\in X}{d_x}$ the maximal degree. The graph is equipped with the usual \emph{combinatorial distance function} $d$ where $d(x,y)$ is the length of shortest path(s) between $x$ and $y$, and the diameter of the graph is defined as $\diam(X)=\max_{x,y\in X} d(x,y)$. In this note, we restrict $Q$ to be a simple random walk, which is given by $Q(x,y)=\frac{1}{d_x}$ for all $y\in N(x)$, and $\pi(x)=d_x/(\sum_{v\in X}{d_v})$.

\section{Proofs}
\label{sec:proof}

\begin{proof}[Proof of Theorem \ref{thm:equivalent}]
	Equivalence $(1) \Leftrightarrow (2)$ was stated in \cite[Theorem 4.5]{EM12} under the assumption that $\theta$ is the logarithmic mean. Equivalence $(2) \Leftrightarrow (3)$ was stated in \cite[Theorem 3.1]{EF18}, regardless of the choice of a `suitable' mean for $\theta$.
\end{proof}

\begin{proof}[Proof of Theorem \ref{thm:local_grad_est}]

The proof follows ideas of \cite[Corollary 3.4]{EF18}. Note that the inequality \eqref{eq:global_grad_est} is homogeneous in $\rho$. Therefore, we can drop the requirement that $\rho$ is a probability density: $\sum_x\rho(x)\pi(x)=1$. We localize inequality \eqref{eq:global_grad_est} by choosing $\rho={\bf{1}}_{x}+\varepsilon \sum_{y\in N(x)} {\bf{1}}_{y}$ for a fixed $x\in V$ and a parameter $\varepsilon\in [0,\infty)$.

In particular we know that $\hat{\rho}(x,y)=\hat{\rho}(y,x)=\theta(1,\varepsilon)$ for all $y\in N(x)$. The left-hand-side of \eqref{eq:global_grad_est} has the following lower bound
\begin{equation} \label{eq:local_grad_est_lhs}
L.H.S. \ge \sum_{y\in N(x)} (P_tf(y)-P_tf(x))^2\theta(1,\varepsilon) Q(x,y)\pi(x)=\theta(1,\varepsilon) \cdot \Gamma(P_tf)(x)\pi(x) 
\end{equation}

On the other hand, we have the following bound on
the right-hand-side of \eqref{eq:global_grad_est} by 
\begin{eqnarray} \label{eq:bound_by_am}
R.H.S. 
&\le& e^{-2\kappa t} \frac{1}{2} \sum_{u,v} (f(v)-f(u))^2 \dfrac{P_t\rho(u)+P_t\rho(v)}{2} Q(u,v)\pi(u) \nonumber \\
&=& e^{-2\kappa t} \frac{1}{2} \sum_{u,v} (f(v)-f(u))^2 P_t\rho(u) Q(u,v)\pi(u)
\end{eqnarray}
due to $\theta(s,t)\le (s+t)/2$ and symmetry from interchanging $u$ and $v$.

We now apply the heat kernel $p_t(\cdot,\cdot)$ given by $P_tg(u)=\sum_{z} p_t(u,z)g(z)\pi(z)$ for every function $g$. With our chosen $\rho$, we obtain $P_t\rho(u)=p_t(u,x)\rho(x)\pi(x)+\varepsilon \sum_{y\in N(x)} p_t(u,y)\rho(y)\pi(y)$, which we substitute into \eqref{eq:bound_by_am} and use the symmetry of heat kernel: $p_t(u,v)=p_t(v,u)$ to derive
\begin{eqnarray} \label{eq:local_grad_est_rhs}
R.H.S. &\le \dfrac{e^{-2\kappa t}}{2} \bigg[& \pi(x)\sum_{u} p_t(u,x)\pi(u) \sum_{v} (f(v)-f(u))^2Q(u,v) + \nonumber \\
& &  \varepsilon\sum_{y\in N(x)} \pi(y) \sum_{u} p_t(u,y)\pi(u) (\sum_{v}(f(v)-f(u))^2Q(u,v)) \bigg] \nonumber \\
&=\dfrac{e^{-2\kappa t}}{2}\bigg[& \pi(x) P_t\Gamma (f)(x)+\varepsilon \sum_{y\in N(x)} \pi(y) P_t\Gamma (f)(y) \bigg].
\end{eqnarray}
The desired inequality then follows from \eqref{eq:local_grad_est_lhs} and \eqref{eq:local_grad_est_rhs}.
\end{proof}

For the underlying graph $X$ with $Q$ representing a simple random walk, we have the following corollary as an immediate consequence of Theorem \ref{thm:local_grad_est}.
\begin{corollary} \label{cor:result_grad_est}
Let $Q$ represent a simple random walk on $X$ with entropic Ricci curvature $\Ric(Q)\ge \kappa$. Then we have the following gradient estimate:
\begin{equation} \label{eq:grad_est_result}
\Gamma(P_t f) (x) \le c\cdot e^{-2\kappa t} \|P_t \Gamma(f)\|_{\infty}
\end{equation}
where $c:=\frac{D\log D}{D-1}$ and $D$ is the maximal degree.
\end{corollary}

\begin{proof}[Proof of Corollary \ref{cor:result_grad_est}]
Theorem \ref{thm:local_grad_est} implies that 
$\Gamma(P_t f) (x) \le c_{\varepsilon,x}\cdot e^{-2\kappa t} \|P_t \Gamma(f)\|_{\infty}$. where $c_{\varepsilon,x}:=\dfrac{1+\varepsilon \sum_{y\in N(x)} \frac{d_y}{d_x}}{2 \theta(1,\varepsilon)} \le \dfrac{1+\varepsilon D}{2\theta(1,\varepsilon)}$. In particular when $\varepsilon=\frac{1}{D}$, we have  $c_{\varepsilon,x}\le c$.
\end{proof}

Finally, we present the proof of the diameter bound in the case of strictly positive entropic curvature.

\begin{proof}[Proof of Theorem \ref{thm:BonnetMyers}]
	
The proof follows ideas of \cite[Theorem 2.1]{LMP18}. Consider a particular choice of function $f\in C(X)$ given by $f(x)=d(x,x_0)$ for an arbitrary reference point $x_0\in X$. Since $f$ is a Lipschitz function with constant $1$, it follows that $\Gamma (f)(x)= \sum_{y\in N(x)} \frac{1}{d_x}(f(y)-f(x))^2\le 1$ for all $x\in X$, i.e., $\|\Gamma (f)\|_\infty \le 1$, which then implies $\| P_t \Gamma(f)\|_\infty \le \| \Gamma(f)\|_\infty \le 1$.

Moreover, Cauchy-Schwartz and inequality \eqref{eq:grad_est_result} give
\begin{eqnarray} \label{eq:GammaPt_bound}
\left|\Delta P_t f(x)\right|^2 
\le \frac{1}{d_x}\sum_{y\in N(x)} (P_tf(y)-P_tf(x))^2
= \Gamma(P_t f)(x)
\le e^{-2\kappa t} c\|P_t\Gamma(f)\|_{\infty}
\le ce^{-2\kappa t}.
\end{eqnarray}
From the fundamental theorem of calculus and the definition of $P_t$, we then obtain $$|f(x)-P_Tf(x)| 
\le \int\limits_{0}^{T} \left|\frac{\partial}{\partial t}P_t f(x)\right|dt
= \int\limits_{0}^{T} \left|\Delta P_t f(x)\right|dt \le \int\limits_{0}^{T} \sqrt{c}e^{-\kappa t} dt \le \frac{\sqrt{c}}{\kappa},$$
which holds true for all $T>0$.

Moreover, the chain of inequalities in \eqref{eq:GammaPt_bound} implies that $|P_tf(y)-P_tf(x)|^2 \le d_x\cdot c e^{-2\kappa t} \rightarrow 0$ as $t\rightarrow \infty$ for all neighbors $y\sim x$. Therefore, $|P_tf(y)-P_tf(x)|\rightarrow 0$ as $t\rightarrow \infty$ for an arbitrary pair of $x,y$ (by considering a connected path from $x$ to $y$).

Passing to the limit $T\rightarrow \infty$, we can conclude from triangle inequality that $$\dd(x,x_0) = |f(x)-f(x_0)| \le |f(x)-P_Tf(x)|+|f(x_0)-P_Tf(x_0)|+|P_Tf(x)-P_Tf(x_0)| \le \dfrac{2\sqrt{c}}{\kappa}.$$
Since $x,x_0$ are arbitrary, we obtain the desired diameter bound.
\end{proof}

\end{document}